\documentclass{amsart}
\usepackage{graphicx}
\usepackage{amssymb}
\usepackage{amsfonts}
\usepackage{hyperref}
\usepackage{mathrsfs}
\swapnumbers
\newtheorem{theorem}{Theorem}[section]
\newtheorem{lemma}[theorem]{Lemma}
\newtheorem{corollary}[theorem]{Corollary}

\theoremstyle{definition}

\newtheorem{assumption}[theorem]{Assumption}
\newtheorem{remark}[theorem]{Remark}

\numberwithin{equation}{section}
 \theoremstyle{plain}
 
 \numberwithin{equation}{section} 
 \numberwithin{figure}{section} 
 \theoremstyle{plain}
 \theoremstyle{plain}
 \theoremstyle{remark}
 \newtheorem*{acknowledgement*}{Acknowledgement}

\newcommand{\cA}{{\mathcal A}}
\newcommand{\cB}{{\mathcal B}}
\newcommand{\cC}{{\mathcal C}}

\newcommand{\cF}{{\mathcal F}}
\newcommand{\cG}{{\mathcal G}}
\newcommand{\cH}{{\mathcal H}}

\newcommand{\cL}{{\mathcal L}}

\newcommand{\cN}{{\mathcal N}}

\newcommand{\cZ}{{\mathcal Z}}
\newcommand{\te}{{\theta}}

\newcommand{\Om}{{\Omega}}

\newcommand{\ve}{{\varepsilon}}
\newcommand{\del}{{\delta}}

\newcommand{\gam}{{\gamma}}
\newcommand{\Gam}{{\Gamma}}

\newcommand{\sig}{{\sigma}}
\newcommand{\al}{{\alpha}}
\newcommand{\be}{{\beta}}
\newcommand{\ka}{{\kappa}}

\newcommand{\bbE}{{\mathbb E}}

\newcommand{\bbN}{{\mathbb N}}

\newcommand{\bbR}{{\mathbb R}}

\newcommand{\bbZ}{{\mathbb Z}}
\newcommand{\bbI}{{\mathbb I}}

\newcommand{\brF}{{\bar F}}

\begin{document}
\title[]{Stein's method for nonconventional sums}%
 \vskip 0.1cm
 \author{Yeor Hafouta \\
\vskip 0.1cm
 Institute  of Mathematics\\
Hebrew University\\
Jerusalem, Israel}%
\address{
Institute of Mathematics, The Hebrew University, Jerusalem 91904, Israel}
\email{yeor.hafouta@mail.huji.ac.il}%

\thanks{ }
\dedicatory{  }
 \date{\today}

\begin{abstract}\noindent
We obtain almost optimal convergence rate in the central limit theorem for
(appropriately normalized) "nonconventional" sums of the form 
$S_N=\sum_{n=1}^N (F(\xi_n,\xi_{2n},...,\xi_{\ell n})-\bar F)$. 
Here $\{\xi_n: n\geq 0\}$ is a sufficiently fast mixing vector
 process with some  stationarity conditions, $F$ is bounded H\"older continuous function and
$\bar F$ is a certain centralizing constant.  Extensions to more general functions $F$ will
be discusses, as well. Our approach here is based on the so called Stein's method, and
the rates obtained in this paper significantly improve the rates in \cite{HK1}.
Our results hold true, for instance, when $\xi_n=(T^nf_i)_{i=1}^\wp$ where $T$ is a topologically 
mixing subshift
of finite type, a hyperbolic diffeomorphism or an expanding transformation
taken with a Gibbs invariant measure, as well as in the case when $\{\xi_n: n\geq 0\}$ forms
a stationary and exponentially fast $\phi$-mixing sequence, which, for instance, holds true
when $\xi_n=(f_i(\Upsilon_n))_{i=1}^\wp$ where $\Upsilon_n$ is a Markov
chain satisfying the Doeblin condition considered as a stationary
process with respect to its invariant measure.

\end{abstract}
\maketitle
\markboth{Y. Hafouta}{Stein's method for nonconventional sums}
\renewcommand{\theequation}{\arabic{section}.\arabic{equation}}
\pagenumbering{arabic}

\section{Introduction}\label{sec1}\setcounter{equation}{0}
Let  $\Phi$ be the
standard normal distribution function and let 
 $X_1,X_2,X_3...$ be a sequence of independent and identically distributed 
random variables such that $\bbE X_1=0$ and $0<\bbE X_1^2=\sig^2<\infty$. 
The classical Berry-Esseen theorem provides a uniform approximation of the error term in the
central limit theorem (CLT) for the sums $\hat S_n=\frac1{\sqrt n\sig}\sum_{k=1}^nX_k$, stating
 that for any $n\in\bbN$,
\begin{equation}\label{ClassB.E.}
\sup_{x\in\bbR}|F_n(x)-\Phi(x)|\leq\frac {C\bbE |X_1|^3}{\sqrt n}
\end{equation}
where $F_n$ is the distribution function of $\hat S_n$
(see Section 6 of Ch. III in \cite{Sh}) and $C>0$ is an absolute constant which by efforts of many
researchers was optimized by now to a number a bit less than 1/2.

During the last 50 years there were several extensions of the CLT for sums of weakly dependent
random variables and for martingales, including many estimates of error terms. Among the most used
methods in the case of weak dependence are Gordin's method for martingale approximation (see \cite{Gor}, \cite{Mc1} and \cite{HH2})
and Stein's method (see \cite{St1}).
 While Stein's method can yield close to optimal convergence rate (see \cite{St1} and \cite{Rin}),
martingale approximation method can not, since Berry-Esseen type estimates for martingales
do not yield (in general) optimal convergence rates even for sums of independent random variables
(see, for instance  \cite{HH2} and \cite{BoltExact}).

 Partially motivated by the research on nonconventional ergodic theorems (the term "nonconventional" 
comes from \cite{Fu}), probabilistic limit theorems for sums of the form 
$S_N=\sum_{n=1}^NF(\xi_{q_1(n)},\xi_{q_2(n)},...,\xi_{q_\ell(n)})$ have 
become a well  studied topic. Here $\{\xi_n, n\geq 0\}$ is a sufficiently fast mixing processes with some 
stationarity properties and $F$ is a function satisfying some regularity conditions.
 The summands here are nonstationary and long range dependent which makes
it difficult to apply standard methods.
 This line of research started  in \cite{Ki2}, where the author proved
a functional CLT for the normalized sums $N^{-\frac12}S_{[Nt]}$ taking the
characteristic function approximation approach.
In \cite{KV} the authors established  a functional CLT for more general $q_i$'s than in \cite{Ki2},
showing that the martingale approximation  approach is applicable. 
Their results included the case when $q_i(n)=in$ which was the original
motivation for the study of nonconventional averages (see \cite{Fu}).
In \cite{HK1} the authors estimated the  convergence rate of $\cZ_N=N^{-\frac12}S_N$ in the Kolmogorov (uniform)
metric towards its weak limit  under the assumptions of \cite{KV}. The proof relied on
 Berry-Esseen type results for martingales, which led to
estimates of order $N^{-\frac1{10}}\ln(N+1)$, which is far from optimal.
In the special case when $\xi_n$'s are independent the authors provided optimal rate of order
$N^{-\frac12}$ relying on Stein's method for sums of locally dependent random variables 
(see \cite{CS}).

The goal of this paper is to show that Stein's method is applicable for nonconventional sums when $\xi_n$'s
are weakly dependent, and to significantly improve the rates obtained in \cite{HK1}.
We first consider the case when $F$ is a bounded H\"older continuous function and
 $q_i(n)=in$ for any $1\leq i\leq\ell$ and $n\in\bbN$, and (in the self normalized case)
provide almost optimal upper bound of the form 
\begin{equation}\label{IntKol}
\sup_{x\in\bbR}|P(S_N\leq x\sqrt{\bbE S_N^2})-\Phi(x)|\leq CN^{-\frac12}\ln^2(N+1)
\end{equation}
assuming that $D^2>0$, where $D^2=\lim_{N\to\infty}\bbE S_N^2$. We also obtain rates of the form
\begin{equation}\label{IntKol2}
\sup_{x\in\bbR}|P(N^{-\frac12}S_N\leq x)-\Phi(xD^{-1})|\leq C_\epsilon N^{-\frac12+\epsilon}
\end{equation}
where $\epsilon>0$ is an arbitrary positive constant and $C_\epsilon$ is a constant
which in general depends on $\epsilon$. 
When $\{\xi_n: n\geq 0\}$ forms
a stationary and exponentially fast $\phi$-mixing sequence then, in fact,  we show that
(\ref{IntKol}) and (\ref{IntKol2}) hold true for any bounded function $F$ which is not necessarily continuous.
Convergence rates for more general functions and more general indexes $q_i(n)$'s will be discussed, as well.

As in \cite{HK1}, our results hold true when, for instance,
$\xi_n=T^nf$ where $f=(f_1,...,f_d)$, $T$ is a mixing subshift
of finite type, a hyperbolic diffeomorphism or an expanding transformation
taken with a Gibbs invariant measure, as well, as in the case when
$\xi_n=f(\Upsilon_n), f=(f_1,...,f_d)$ where $\Upsilon_n$
is a Markov chain satisfying the Doeblin condition considered as a
stationary process with respect to its invariant measure. In fact, 
any stationary and exponentially 
fast $\phi$-mixing sequence $\{\xi_n\}$ can be considered.
In the dynamical systems case each $f_i$ should be either H\" older
continuous or piecewise constant on elements of Markov partitions.
As an application we can consider  $\xi_n=((\xi_n)_1,...,(\xi_n)_\ell)$,
$(\xi_n)_j=\bbI_{A_j}(T^n x)$ in the dynamical systems case and
$(\xi_n)_j=\bbI_{A_j}(\Upsilon_n)$ in the Markov chain case where
$\bbI_{A}$ is the indicator of a set $A$. Let  $F=F(x_1,...,x_\ell)$,
$x_j=(x_j^{(1)},...,x_j^{(\ell)})$ be a bounded H\"older continuous function which identifies with 
the function $G(x_1,...,x_\ell)=x_1^{(1)}\cdot x_2^{(2)}\cdots x_\ell^{(\ell)}$ on the cube $([0,1]^\wp)^\ell$.
Let $N(n)$ be the number of $l$'s
between $0$ and $n$ for which $T^{q_{j}(l)}x\in A_j$ for $j=0,1,...,\ell$
(or $\Upsilon_{q_{j}(l)}\in A_j$ in the Markov chains case), where we set $q_{0}=0$, namely
the number of $\ell-$tuples of return times to
$A_j$'s (either by $T^{q_j(l)}$ or by $\Upsilon_{q_j(l)}$). Then
our results yield a central limit theorem with almost optimal convergence rate for the numbers $N(n)$.

\section*{Acknowledgement}
This paper is a part of the author's PhD thesis conducted at the Hebrew university of Jerusalem. 
I would like to thank my advisor Professor Yuri Kifer for suggesting to me the problem studied in this 
paper and for many helpful discussions.

\section{Preliminaries and main results}
\label{sec2}\setcounter{equation}{0}
Our setup consists of a $\wp$-dimensional stochastic process $\{\xi_n,n\geq0\}$
on a probability space $(\Omega,\cF,P)$ and a family
of sub $\sig-$algebras $\cF_{k,l}$, $-\infty\leq k\leq l\leq\infty$
such that $\cF_{k,l}\subset\cF_{k',l'}\subset\cF$ if $k'\leq k$
and $l'\geq l$. We will impose restrictions of the mixing coefficients
\begin{equation}\label{MixCoef1}
\phi(n)=\sup\{\phi(\cF_{-\infty,k},\cF_{k+n,\infty}): k\in\bbZ\}
\end{equation}
where we recall that for any two sub $\sigma-$algebras $\cG,\cH\subset\cF$, 
\[
\phi(\cG,\cH)=\sup\big\{\big|\frac{P(A\cap B)}{P(A)}-P(B)\big|: A\in\cG, B\in\cH,
 P(A)>0\}.
\]

In order to ensure some applications,
 in particular, to dynamical systems we will not assume that $\xi_n$
is measurable with respect to $\cF_{n,n}$ but instead impose conditions
on the approximation rates
\begin{equation}\label{2.3}
\beta_\infty(r)=\sup_{k\geq0}\|\xi_k-\bbE[\xi_k|\cF_{k-r,k+r}]\|_{L^\infty}
\end{equation}
where  $\|X\|_{L^\infty}$ denotes the essential supremum of the absolute value of a random variable $X$.

We do not require stationarity of the process $\{\xi_n, n\geq 0\}$, 
assuming only that the distribution of $\xi_n$ does not depend on $n$  and that
the joint distribution of $(\xi_n,\xi_m)$ depends
only on $n-m$ which we write for further
reference by
\begin{equation}\label{2.10}
\xi_n\thicksim\mu\,\,\,\text{ and }\,\, \big(\xi_n,\xi_m\big)\thicksim
\mu_{m-n}
\end{equation}
where $Y\thicksim\mu$ means that $Y$ has $\mu$ for its
distribution. 

Let $F=F(x_1,...,x_\ell):(\bbR^{\wp})^\ell\to\bbR,\,\ell\geq1$ be a bounded H\"older function
 and let $M>0$  and $\ka\in(0,1]$ be such that 
\begin{eqnarray}\label{GoodHold}
|F(x)|\leq M\,\,\text{ and }\,\,\,\,\,\,\,\,\,\,\,\,\,\\
|F(x)-F(y)|\leq M\sum_{i=1}^\ell|x_i-y_i|^\ka\label{GoodHold1}
\end{eqnarray}
for any $x=(x_1,...,x_\ell)$ and $y=(y_1,...,y_\ell)$ in $(\bbR^{\wp})^\ell$.
To simplify  formulas we assume the centering condition
\begin{eqnarray}\label{2.8}
\brF =\int F(x_1,x_2,...,x_\ell)d\mu(x_1)d\mu(x_2)... d\mu(x_\ell)=0
\end{eqnarray}
 which is not really a restriction since
we can always replace $F$ by $F-\brF$. The main goal of this paper is to prove a central limit theorem
with convergence rate for the normalized sums $(c_N)^{-1}S_N$, where
\[
S_N=\sum_{n=1}^NF(\xi_n,\xi_{2n},...,\xi_{\ell n})
\]
and either $c_N=N^{-\frac12}$ or $c_N=\sqrt{\bbE S_N^2}$.
\begin{assumption}\label{exp ass}
There exist $d>0$ and $c\in(0,1)$ such that 
\begin{equation}\label{exp psi}
\phi(n)+(\beta_\infty(n))^\ka \leq dc^n
\end{equation} 
for any $n\in\bbN$. 
\end{assumption}

The following theorem is a consequence of the arguments in
\cite{KV}, \cite{Ki3} and  \cite{HK1} and is formulated here for 
readers' convenience.
\begin{theorem}\label{D-thm}
 Suppose that Assumption
(\ref{exp ass}) is satisfied. Then the  limit $D^2=\lim_{N\to\infty}N^{-1}\bbE S_N^2$ exists and there exists 
$C_1>0$ which depends only on $\ell,c$ and $d$
such that 
\begin{equation}\label{D-Rate}
|\bbE S_N^2-D^2N|\leq C_1M^2N^{\frac12}
\end{equation}
for any $N\in\bbN$. Moreover,
 $D^2>0$  if and only if there exists no stationary in the wide sense process $\{V_n: n\geq 1\}$
such that 
\[
F(\xi^{(1)}_n,\xi^{(n)}_{2n},...,\xi_{\ell n}^{(\ell)})=V_{n+1}-V_n,\,\text{ P-almost surely}
\]
for any $n\in\bbN$, where $\xi^{(i)}$, $i=1,...,\ell$ are independent copies of $\xi=\{\xi_n: n\geq1\}$.
\end{theorem}

Next, recall that the Kolmogorov (uniform) metric is defined for each pairs of  distributions $\cL_1$ and $\cL_2$ on
$\bbR$ with distribution functions $G_1$ and $G_2$ by
\begin{eqnarray*}
d_K(\cL_1,\cL_2)=\sup_{x\in\bbR}|G_1(x)-G_2(x)|.
\end{eqnarray*}
For any random variable $X$ we denote its law by $\cL(X)$.
Our main result is the following theorem.
\begin{theorem}\label{thm3.3}
Suppose that Assumption (\ref{exp ass}) holds true and  that $D^2>0$. Set $s_N=\sqrt{\bbE S_N^2}$
and $Z_N=(s_N)^{-1}S_N$ when $s_N>0$ while when $s_N=0$ we set $Z_N=N^{-\frac12}S_N$.
 Let $\cN(0,1)$ be the zero mean normal
distribution with variance  $1$.  
 Then  there exists a constant $C>0$ which depends  
 only on $\ell,d$ and $c$ such that 
\begin{equation}\label{Kol-Opt}
d_K(\cL(Z_N),\cN(0,1))
\leq C\max(1,\rho^3)N^{-\frac12}\ln^2(N+1)
\end{equation}
for any $N\in\bbN$, where  $\rho=M D^{-1}$. Moreover, for any $\epsilon>0$ there
exists a constant $c_\epsilon>0$ which depends only on $\epsilon,c,d$ and $\ell$ so that for any $N\geq 1$, 
\begin{equation}\label{Kol-Opt1.1.}
d_K(\cL(N^{-\frac12}S_N),\cN(0,D^2))
\leq c_\epsilon \max(1,\rho^3)
N^{-\frac12+\epsilon}
\end{equation}
where $\cN(0,D^2)$ is the zero mean normal
distribution with variance  $D^2$.
When $\be_\infty(r_0)=0$ for some $r_0$
then (\ref{Kol-Opt}) and (\ref{Kol-Opt1.1.}) hold true  with constants $C$ and $c_\epsilon$ 
which depend also on $r_0$,
assuming only that $F$ is a bounded function satisfying (\ref{GoodHold}).
\end{theorem}
Note that $\be_\infty(0)=0$ when $\cF_{m,n}=\sig\{\xi_{\max(0,m)},...,\xi_{\max(0,n)}\}$ and
therefore when the processes $\{\xi_n: n\geq 0\}$ itself is exponentially
fast $\phi$-mixing (i.e. when (\ref{exp psi}) holds true with these $\sig-$algebras)
 we obtain (\ref{Kol-Opt}) for any bounded function $F$.

The outline of the proof goes as follows. Relying on \cite{Rin}, 
Stein's method becomes effective  for the sum $S_N$
 when $\{F_n: 1\leq n\leq N\}$,  $F_n=F(\xi_n,\xi_{2n},...,\xi_{\ell n})$
 are locally weak dependent  in the sense that there exist sets $A_n$  and
nonnegative  integers $d_{n,m}$, $1\leq n,m\leq N$ such that $n\in A_n$
$a_n=|A_n|$ and $b_n(k)=|\{1\leq m\leq N: d_{n,m}=k\}|,\,k\geq 0$ are small relatively to $N$,
$F_n$ and $\{F_s: s\not\in A_n\}$ are weakly dependent 
 and the random vectors $\textbf F_n=\{F_k: k\in A_n\}$ and $\textbf F_m=\{F_s: s\in A_m\}$
are weakly dependent when $d_{n,m}$ is sufficiently large.  
We first reduce the problem of approximation of 
the left hand side of (\ref{Kol-Opt}) to the case when $\xi=\{\xi_n: n\geq 0\}$ forms a sufficiently fast
$\phi$-mixing process. Then we consider the sets
\[
A_n=A_{n,N,l}=\{1\leq m\leq N: \min_{1\leq i,j\leq\ell}|in-jm|\leq l\}
\]
and the numbers $d_{n,m}=\min\{|ia-jb|: a\in A_n, b\in A_m, 1\leq i,j\leq\ell\}$
and show that $a_n$ and $b_n(k)$ defined above are of order $l$. 
In  Section \ref{sec3} we provide estimates which will show that the required type of the above 
weak dependence is satisfied, and then we take $l$ of the form $l=A\ln(N+1)$ to complete
the proof. 
 In fact, existing estimates on the left hand side of (\ref{Kol-Opt}) using Stein's method become
effective only after using
the expectation estimates obtained in Section \ref{sec3} even  for "conventional" sums of $\phi$-mixing sequences 
(i.e. in the case $\ell=1$), which is a particular
case of our setup, and so, in particular, we show that Stein's method is effective for such sums and
yields almost optimal convergence rate.

\section{Auxiliary results}\label{sec3}\setcounter{equation}{0}
The following result will be used.
\begin{lemma}\label{lem3.17} 
Let $X$ and $Y$ be two random variables defined on the same probability space. 
Let $Z$ be a random variable with density $\rho$ bounded from above by some constant $c>0$. Then,
\begin{eqnarray*}
d_K(\cL(Y),\cL(Z))\leq 3d_K(\cL(X),\cL(Z))+
4c\|X-Y\|_{L^\infty}\,\,\text{ and for any }\,\, b\geq 1, \\
d_K(\cL(Y),\cL(Z))\leq 3d_K(\cL(X),\cL(Z))+(1+4c)\|X-Y\|_{L^b}^{1-\frac 1{b+1}}.
\end{eqnarray*}
\end{lemma}
The second inequality is proved in Lemma 3.3 in \cite{HK2}, while 
the proof of the first inequality
goes in the same way as the proof of that
Lemma 3.3, taking in (3.2) from there $\del=\|X-Y\|_{L^\infty}$.

Next, we  recall  that  (see \cite{Br}, Ch. 4) for any two sub 
$\sigma-$algebras $\cG,\cH\subset\cF$,
\begin{equation}\label{Phi-Rel}
2\phi(\cG,\cH)=\sup\{\|\bbE[g|\cG]-\bbE g\|_{L^\infty}\,:  g\in L^\infty(\Om,\cH,P),\,
\|g\|_{L^\infty}\leq1\}.
\end{equation}
The following lemma does not seem to be new but  for reader's convenience and completeness
we will sketch its proof here.
\begin{lemma}\label{thm3.11}
Let $\cG_1,\cG_2\subset\cF$ be two sub $\sigma-$algebras of $\cF$ and for $i=1,2$ let
$V_i$ be a $\bbR^{d_i}$-valued random $\cG_i$-measurable vector with distribution $\mu_i$.  Set
$d=d_1+d_2$, $\mu=\mu_1\times\mu_2$, denote by $\ka$
the distribution of the random vector $(V_1,V_2)$
and consider the measure $\nu=\frac12(\ka+\mu)$.  Let $\cB$  be the Borel $\sigma-$algebra on $\bbR^d$ and 
$H\in L^\infty(\bbR^d,\cB,\nu)$. Then $\bbE[H(V_1,V_2)|\cG_1]$ and $\bbE H(v,V_2)$ exist for $\mu_1$-almost
any $v\in\bbR^{d_1}$ and
\begin{equation}\label{Meas}
|\bbE[H(V_1,V_2)|\cG_1]-h(V_1)|
\leq2\|H\|_{L^\infty(\bbR^d,\cB,\nu)}\phi(\cG_1,\cG_2),\,\,P-a.s.
\end{equation}
 where $h(v)=\bbE H(v,V_2)$ and a.s. stands for almost surely.
\end{lemma}

\begin{proof}
Clearly $H$ is bounded $\mu$ and $\ka$ a.s.. Thus  
$\bbE[H(V_1,V_2)|\cG_1]$ exists  and existence of $\bbE H(v,V_2)$ ($\mu_1$-a.s.)
follows from the Fubini theorem.
Relying on (\ref{Phi-Rel}), inequality (\ref{Meas})
follows easily for functions 
of the form $G(v_1,v_2)=\sum_iI(v_1\in A_i)g_i(v_2)$ where $\{A_i\}$
is a measurable partition of the support of $\mu_1$. Any uniformly continuous function
$H$ is a uniform limit of functions of the above form, which implies that (\ref{Meas}) 
holds true for uniformly continuous functions. 
Finally, by Lusin's theorem (see \cite{Ru}), 
any  function $H\in L^\infty(R^d,\cB,\nu)$ is an $L^1$ (and a.s.) limit of a sequence 
$\{H_n\}$ of continuous functions with compact support satisfying  
$\|H_n\|_{L^\infty(\bbR^d,\cB,\nu)}\leq\|H\|_{L^\infty(\bbR^d,\cB,\nu)}$
 and (\ref{Meas}) follows for any  $H\in L^\infty(\bbR^d,\cB,\nu)$.
\end{proof}

\begin{corollary}\label{lem3.1}
Let $U_i$ be a $d_i$-dimensional random vector, $i=1,...,k$ defined on the
probability space $(\Om,\cF,P)$ from Section \ref{sec2}. Suppose that 
each $U_i$ is $\cF_{m_i,n_i}$-measurable, where $n_{i-1}<m_i\leq n_i<m_{i+1}$, 
$i=1,...,k$, $n_0=-\infty$ and $m_{k+1}=\infty$.
Let $\{\cC_i: 1\leq i\leq s\}$ be a  partition of $\{1,2,...,k\}$. Denote by $\mu_i$ the
distribution of the random vector $U(\cC_i)=\{U_j: j\in\cC_i\}$, $i=1,...,s$. 
Then, for any bounded
Borel function $H:\bbR^{d_1+d_2+...+d_k}\to\bbR$,
\begin{eqnarray}\label{Expec1}
&\,\,\,\,\,\,\,\,\,\big|\bbE H(U_1,U_2,...,U_k)-\int H(u_1,u_2,...,u_k)d\mu_1
(u^{(\cC_1)})d\mu_2(u^{(\cC_2)})...d\mu_s(u^{(\cC_s)})\big|\\
&\leq4\|H\|_\infty\sum_{i=2}^k\phi(m_i-n_{i-1})\nonumber
\end{eqnarray}
where $u^{(\cC_i)}=\{u_j: j\in\cC_i\}$, $i=1,...,s$ and $\|H\|_\infty$ stands for
the supremum of $H$. Namely, let
$U^{(i)}(\cC_i)$ be independent copies of the processes  $U(\cC_i)$, $i=1,...,s$.
Then
\begin{eqnarray*}
&\big|\bbE H(U_1,U_2,...,U_k)-\bbE H(U_1^{(j_1)},U_2^{(j_2)},...,U_k^{(j_k)})\big|\leq
4\|H\|_\infty\sum_{i=2}^k\phi(m_i-n_{i-1})
\end{eqnarray*}
where  $j_i$ satisfies that $i\in\cC_{j_i}$, for any $1\leq i\leq k$.
\end{corollary}

\begin{proof}
Denote by $\nu_i$ the distribution of $U_i$, $i=1,..,k$.
We first prove by induction on $k$ that for any choice of $H$ and $U_i$'s with 
the required properties,
\begin{eqnarray}\label{Expec2}
&|\bbE H(U_1,U_2,...,U_k)-\int H(u_1,u_2,...,u_k)d\nu_1(u_1)d\nu_2(u_2)...d\nu_k(u_k)|\leq\\
&2\|H\|_\infty\sum_{i=2}^v\phi(m_i-n_{i-1}).\nonumber
\end{eqnarray}
Indeed, suppose that  $k=2$ and set  $V_1=U_1$, $V_2=U_2$, $h(u_1)=E[H(u_1,U_2)]$, 
$\cG_1=\cF_{-\infty,n_1}$ and $\cG_2=\cF_{m_2,\infty}$.
Taking expectation  in (\ref{Meas}) yields
\[
|\bbE H(U_1,U_2)-\bbE h(U_1)|\leq 2\|H\|_\infty\phi(m_2-n_1)
\]
which means that (\ref{Expec2}) holds true when $k=2$. Now, suppose
that (\ref{Expec2}) holds true for any $k\leq j-1$, $U_1,...,U_k$ with the required properties
 and any bounded Borel function $H:\bbR^{e_1+...+e_{k-1}}\to\bbR$, where $e_1,...,e_{k-1}\in\bbN$.
In order to deduce (\ref{Expec2}) for $k=j$,
set $V_1=(U_1,...,U_{j-1})$, $V_2=U_j$, $h(v_1)=\bbE H(v_1,U_j)$, $v_1=(u_1,...,u_{j-1})$, 
 $\cG_1=\cF_{-\infty,n_{j-1}}$ and $\cG_2=\cF_{m_j,\infty}$.
Taking expectation in (\ref{Meas}) yields 
\[
|\bbE H(U_1,U_2,...,U_j)-\bbE h(U_1,U_2,...,U_{j-1})|\leq 2\|H\|_\infty\phi(m_j-n_{j-1}).
\]
Applying the induction hypothesis with
the function $h$ completes the proof of  (\ref{Expec2}), since  $\|h\|_\infty\leq\|H\|_\infty$.
Next, we prove by induction on $s$ that for any choice of $k$, $H$, $U_i$'s with the required properties
and $\cC_1,...,\cC_s$,
\begin{eqnarray}\label{Exec3}
&\big|\int H(u_1,u_2,...,u_k)d\mu_1
(u^{(\cC_1)})d\mu_2(u^{(\cC_2)})...d\mu_s(u^{(\cC_s)})-\\
&\int H(u_1,u_2,...,u_k)d\nu_1(u_1)d\nu_2(u_2)...d\nu_k(u_k)\big|\leq
2\|H\|_\infty\sum_{i=2}^k\phi(m_i-n_{i-1}).\nonumber
\end{eqnarray}
 For $s=1$ this is just (\ref{Expec2}). Now suppose that (\ref{Exec3}) 
holds true for any $s\leq j-1$, and any real valued bounded Borel function $H$
defined on $\bbR^{d_1+...+d_k}$, where  $k$ and $d_1,...,d_k$ are some natural numbers. In order to prove 
 (\ref{Exec3})  for $s=j$, 
set $u^{(I)}=(u^{(\cC_1)},u^{(\cC_2)},...,u^{(\cC_{s-1})})$ and let the function $I$ be defined by
\begin{eqnarray}\label{I-fun}
 &I(u^{(I)})= \int H(u_1,u_2,....,u_k)\prod_{j\in\cC_s}d\nu_j(u_j).
\end{eqnarray}
Then
\begin{eqnarray}\label{Step}
&\int H(u_1,u_2,...,u_k)d\nu_1(u_1)d\nu_2(u_2)...d\nu_k(u_k)
=\int  I(u^{(I)})\prod_{j\not\in\cC_s}d\nu_j(u_j).
\end{eqnarray}
Let the function $J$ be defined by
\begin{eqnarray}\label{J-fun}
 &J(u^{(I)})= \int H(u_1,u_2,....,u_k)d\mu_s(u^{(\cC_s)}).
\end{eqnarray} 
Then by (\ref{Expec2}), for any $u^{(\cC_1)},...,u^{(\cC_{s-1})}$,
\begin{eqnarray}\label{I,J}
&|I(u^{(I)})-J(u^{(I)})|\leq
2\|H\|_\infty\sum_{i\in\cC_s}\phi(m_i-n_{i-1}).
\end{eqnarray}
 It is clear that $\|J\|_\infty\leq\|H\|_\infty$. 
Applying the induction hypothesis with the  function $J$ (considered as a function of the variable $u$)
and taking into account (\ref{Step}) and (\ref{I,J})  we obtain
(\ref{Exec3}) with $s=j$. We complete the induction. Inequality (\ref{Expec1}) follows by (\ref{Expec2}) and 
(\ref{Exec3}), and the proof of Corollary \ref{lem3.1} is complete.
\end{proof}

\begin{remark}\label{Rem}
In the notations of Corollary \ref{lem3.1}, let $Z_i, i=1,...,s$  be a bounded  
$\sig\{U(\cC_i)\}$-measurable random variable. Then each $Z_i$ has the form
$Z_i=H_i\big(U(\cC_i)\big)$ for some function $H_i$ which satisfies 
$\|H_i\|_\infty\leq\|Z_i\|_{L^\infty}$. By considering the function 
$H(u)=\prod_{i=1}^sH_i(u^{(\cC_i)})$, we obtain from (\ref{Expec1}) that,
\begin{eqnarray}\label{Exp4}
&\big|\bbE[\prod_{i=1}^sZ_i]-\prod_{i=1}^s\bbE Z_i\big|\leq 4\big(\prod_{i=1}^s\|Z_i\|_{L^\infty}\big)
\sum_{j=2}^k\phi(m_j-n_{j-1}).
\end{eqnarray}
In general we can replace $\|H\|_\infty$ appearing in the right hand side of (\ref{Expec1})
by some essential supremum norm of $H$ with respect to some measure which
has a similar but more complicated form as $\ka$ defined in Lemma \ref{thm3.11}.
\end{remark}

\section{Nonconventional CLT with almost optimal convergence rate via Stein's method}
\label{sec4}\setcounter{equation}{0}

First, the proof of Theorem \ref{D-thm} follows from arguments in  \cite{KV}, \cite{Ki3}, and 
 \cite{HK1}. Indeed,  relying on (2.25) in \cite{KV}, the conditions of \cite{HK1} and \cite{Ki3}
 hold true in our circumstances. Existence of $D^2$ follows from Theorem 2.2 in \cite{KV}, 
inequality (\ref{D-Rate}) follows from the arguments in \cite{Ki3} (first by considering the case when $M=1$)
 and the condition for positivity follows from Theorem 2.3 in
\cite{HK1}. 

Before proving Theorem \ref{thm3.3} we introduce the
following  notations. For any $a,b\in\bbR$ set 
\begin{eqnarray*}
d_\ell(a,b)=\min_{1\leq i,j\leq \ell}|ia-jb|
\end{eqnarray*}
 and for any
 $A,B\subset\bbR$ set
\begin{equation*}
dist(A,B)=\inf\{|a-b|: a\in A,b\in B\}\text{ and }
d_\ell(A,B)=\inf\{d_\ell(a,b): a\in A,b\in B\}.
\end{equation*}
Finally, for any $A_1,A_2,...,A_L\subset \bbR$, 
we will write $A_1<A_2<...<A_L$ if $a_1<a_2<...<a_L$ for any $a_i\in A_i$, $i=1,2,...,L$.

\subsection*{Proof of Theorem \ref{thm3.3}}
Suppose that $D^2>0$. We consider first the self normalized case.
Clearly, in the proof of (\ref{Kol-Opt})
we can assume that $M=1$.
For any 
$N\geq 1$ set $s_N=\sqrt{\bbE (S_N)^2}$. Then by (\ref{D-Rate}),
\begin{equation}\label{sN2 apprx }
(s_N)^2\geq D^2N-N^{\frac12}C_1.
\end{equation}
Let $N$ be so that $D^2N>N^{\frac12}C_1$. Then 
$s_N>0$ and we set 
$Z_N=\frac{S_N}{s_N}$. Let $l$ be of the form $l=4A\ln (N+1)$ where $A\geq1$ is 
a positive constant considered here as a parameter which will be chosen later.  
Set $r=[\frac l4]$ and 
\begin{equation*}
S_{N,r}=\sum_{n=1}^N F(\xi_{n,r},\xi_{2n,r},...,\xi_{\ell n,r})
\end{equation*} 
where $\xi_{m,r}=\bbE[\xi_m|\cF_{m-r,m+r}]$ for any $m\in\bbN$.
Then by (\ref{GoodHold1}) and (\ref{exp psi}),
\begin{equation}\label{I.1}
\|S_N-S_{N,r}\|_{L^\infty}\leq \ell N(\be_\infty(r))^\ka\leq
d\ell c^{-1}Nc^{\frac l4}=d\ell c^{-1}Nc^{A\ln(N+1)}\leq c_0(N+1)^{1+A\ln c} 
\end{equation}
where $c_0=d\ell c^{-1}$ and we also used our assumption that $M=1$. 
Next, let $n>l$, consider the random 
vectors $U_i=\xi_{in,r}$ and set $m_i=in-r$ and $n_i=in+r$, $i=1,...,\ell$.
 Then each $U_i$ is $\cF_{in-r,in+r}$-measurable and $m_i-n_{i-1}=n-2r\geq l-2r\geq \frac l2$.
Applying Corollary \ref{lem3.1} with the sets $\cC_i=\{in\}, i=1,...,\ell$
 we obtain
\begin{eqnarray*}
\big|\bbE F(\xi_{n,r},\xi_{2n,r}...,\xi_{\ell n,r})-
\bbE F(\xi_{n,r}^{(1)},\xi_{2n,r}^{(2)},...,\xi_{\ell n,r}^{(\ell)})\big|\\\leq
4\ell\phi(\frac l2)\leq 4\ell dc^{\frac l2}=c_1(N+1)^{2A\ln c}
\end{eqnarray*}
where $c_1=4\ell d$ and $\xi_{in,r}^{(i)}$'s are independent copies of $\xi_{in,r}$'s. Considering the
product measure of the laws of the vectors $(\xi_{in,r},\xi_{in}), i=1,...,\ell$, 
we can always assume that there (on a larger probability space) exist independent copies
 $\xi_{in}^{(i)}$ of the $\xi_{in}$'s such that 
$\|\xi_{in}^{(i)}-\xi_{in,r}^{(i)}\|_{L^\infty}\leq\beta_\infty(r)$ for any $i=1,2,...,\ell$. 
Thus by (\ref{GoodHold1}) and (\ref{exp psi}),
\begin{eqnarray*}
|\bbE F(\xi_{n,r}^{(1)},\xi_{2n,r}^{(2)},...,\xi_{\ell n,r}^{(\ell)})-
\bbE F(\xi_n^{(1)},\xi_{2n}^{(2)},...,\xi_{\ell n}^{(\ell)})|\\\leq \ell\big(\beta_\infty(r)\big)^\ka\leq
\ell dc^{-1}c^{\frac l4}=c_0(N+1)^{A\ln c} 
\end{eqnarray*}
and notice that $\bbE F(\xi_n^{(1)},\xi_{2n}^{(2)},...,\xi_{\ell n}^{(\ell)})=\bar F=0$. 
We conclude from (\ref{GoodHold}) and the above estimates  that
\begin{eqnarray}\label{I'}
|\bbE S_{N,r}|\leq |\bbE S_{l,r}|+
N(4\ell dc^{\frac l2}+d\ell c^{-1}c^{\frac l4})\\\leq
2l+5N\ell d c^{-1}c^{\frac l4}\leq 8A\ln(N+1)+5c_0(N+1)^{1+A\ln c}.\nonumber
\end{eqnarray}
We assume henceforth that $-A\ln c=A|\ln c|>2$ and set
$\bar S_{N,r}=S_{N,r}-\bbE S_{N,r}$. 
For any two random variables $X$ 
and $Y$ defined on the same probability space we have $|\bbE X^2-\bbE Y^2|\leq\|X+Y\|_{L^2}\|X-Y\|_{L^2}$
and therefore by (\ref{I.1}) and (\ref{D-Rate}),
\begin{eqnarray*}
|\bbE S_N^2-\bbE (S_{N,r})^2|\leq(2\|S_N\|_2+c_0(N+1)^{1+A\ln c})c_0(N+1)^{1+A\ln c}\\
\leq 3c_0(2+c_0+C_1+D)(N+1)^{\frac32+A\ln c}
\end{eqnarray*}
where we also used that $A|\ln c|>1$.
Next, by (\ref{I'}),
\begin{eqnarray*}
|\bbE S_{N,r}^2-\bbE (\bar S_{N,r})^2|=|\bbE S_{N,r}^2-\text{Var}S_{N,r}|=
|\bbE S_{N,r}|^2\\\leq 32A^2\ln^2(N+1)+25c_0^2(N+1)^{2+2A\ln c}
\end{eqnarray*}
and together with the previous inequality and our assumption that $A\ln c<-2$
we obtain that 
\begin{equation}\label{Yielding that}
|\bbE S_N^2-\bbE (\bar S_{N,r})^2|\leq c_2\ln^2(N+1)
\end{equation}
where $c_2=32A^2+25c_0^2+3c_0(2+c_0+C_1+D)$. Combining this wih (\ref{sN2 apprx }), it follows that 
\begin{equation}\label{Using...}
\bbE (\bar S_{N,r})^2\geq D^2N-N^{\frac12}C_1-c_2\ln^2(N+1).
\end{equation}
Let $N$ be so large such that the right hand side in the previous inequality is positive. Then
$\bbE (\bar S_{N,r})^2>0$. Let $\bar s_{N,r}$ be its 
positive square root and set 
 \[
\bar Z_{N,r}=\frac{\bar S_{N,r}}{\bar s_{N,r}}=\sum_{n=1}^NY_n
\]
where for each $n$,
\[
Y_n=Y_{n,N,r}=\frac{F(\xi_{n,r},\xi_{2n,r},...,\xi_{\ell n,r})-
\bbE F(\xi_{n,r},\xi_{2n,r},...,\xi_{\ell n,r})}
{\bar s_{N,r}}.
\]
Observe now that
\begin{eqnarray*}
\|Z_N-\bar Z_{N,r}\|_{L^\infty}\leq\|(s_N)^{-1}S_N-(\bar s_{N,r})^{-1}S_{N,r}\|_{L^\infty}+
|(\bar s_{N,r})^{-1}\bbE S_{N,r}|\leq\\
(s_N)^{-1}\|S_N-S_{N,r}\|_{L^{\infty}}
+|(s_N)^{-1}-(\bar s_{N,r})^{-1}|\|S_{N,r}\|_{L^\infty}+
|(\bar s_{N,r})^{-1}\bbE S_{N,r}|.
\end{eqnarray*}
The inequality $|x^{-1}-y^{-1}|=|x^2-y^2|(x^2y+y^2x)^{-1}$ holds true
for any $x,y>0$ yielding that
\[
|(s_N)^{-1}-(\bar s_{N,r})^{-1}|\leq\frac{c_2\ln^2(N+1)}{(s_N+\bar s_{N,r})s_N\bar s_{N,r}}:=e_2
\]
where we used (\ref{Yielding that}), and 
we conclude from (\ref{I.1}), (\ref{I'})  and the above estimates that 
\begin{eqnarray}\label{Good apprx 0}
\|Z_N-\bar Z_{N,r}\|_{L^\infty}\leq (s_N)^{-1}c_0(N+1)^{1+A\ln c}+2Ne_2+\\
(\bar s_{N,r})^{-1}
(8A\ln(N+1)+5c_0(N+1)^{1+A\ln c})\nonumber
\end{eqnarray}
where we used that $\|S_{N,r}\|_{L^\infty}\leq 2N$ (recall our assumption that $M=1$).
Next, using (\ref{sN2 apprx }), (\ref{Using...}) and that $\ln(N+1)\leq N^{\frac12}$
for any $N\geq 1$ we derive that 
$\min(s_N^2,\bar s_{N,r}^2)\geq \frac14D^2N$ when $3N^{\frac12}D^2\geq 8(C_1+c_2)$
and in this case 
\begin{equation}\label{Good apprx 1}
\|Z_N-\bar Z_{N,r}\|_{L^\infty}\leq c_4\max(D^{-1},D^{-3})N^{-\frac12}\ln^2(N+1)
\end{equation}
where $c_4=C_4(1+c_0+c_2+A)$, $C_4>1$ is some absolute constant and we also used that $N^{1+A\ln c}<1$ .

Next, let $N$ be sufficiently large so that  $3N^{\frac12}D^2\geq 8(C_1+c_2)$.
Then  by (\ref{GoodHold}) and the above lower bound of $(\bar s_{N,r})^2$,
\begin{eqnarray}\label{Y-Bound}
\|Y_n\|_{L^\infty}\leq 2(\bar s_{N,r})^{-1}\leq 4D^{-1}N^{-\frac12}.
\end{eqnarray}
For any $n=1,2,...,N$ set
\[
A_n=A_{n,l,N}=\{1\leq m\leq N: \min_{1\leq i,j\leq\ell}|in-jm|\leq l\}=\{1\leq m\leq N: d_\ell(n,m)\leq l\}
\]
and for any $k\geq 0$ set 
\[
\cA_n(k)=\{1\leq m\leq N: d_\ell(A_n,A_m)=k\}.
\]
We claim that there exist constants $K_1$ and $K_2$ which depend only on $\ell$ such that for 
any $n=1,2,...,N$ and $k\geq 0$, 
\begin{eqnarray}\label{A-Bound}
&|A_n|\leq K_1l\,\,\text{ and }\,\,|\cA_n(k)|\leq K_2l.
\end{eqnarray}
Indeed,
since $A_n$ is contained in a union of at most $\ell^2$ intervals whose lengths do not exceed $2l+1$
we have $|A_n|\leq\ell^2(2l+1)$ and since $l\geq 1$ we can take $K_1=3\ell^2$. To prove the second
inequality, let $n$ and $m$ be such that $d_\ell(A_n,A_m)=k\geq 0$. 
Then there exist $1\leq i_s,j_s\leq\ell$, $s=1,2,3$
and $1\leq u,v\leq N$ such that $|i_3u-j_3v|=k$, $|i_1n-j_1u|\leq l$ and $|i_2m-j_2v|\leq l$.
When $j_3v-i_3u_3=k$, we deduce from the last two inequalities that 
\begin{eqnarray*}
\big|m-\frac{j_2i_3i_1}{i_2j_3j_1}n-\frac{j_2}{j_3i_2}k\big|\leq l\big
(\frac1{i_2}+\frac{j_2i_3}{i_2j_3j_1}\big)
\end{eqnarray*}
and similar inequality holds when $j_3v-i_3u_3=-k$. Thus, when $n$ and $k$ are fixed 
the set $\cA_n(k)$ is contained in a union of $2\ell^6$ intervals whose lengths do not exceed
$2(\ell^2+1)l$, and the choice of $K_2=4\ell^6\cdot (\ell^2+2)$ is sufficient.

Now, set $\del=\del_{l,N}=\sum_{n=1}^N\sum_{m\in A_n}\bbE Y_nY_m$.
 Then
\begin{eqnarray}\label{Q2}
1=\text{Var}\bar Z_{N,r}=\bbE \big(\sum_{n=1}^NY_n\big)^2=\del+\gam
\end{eqnarray}
where 
$\gam=\gam_{l,N}=\sum_{n=1}^N\sum_{m\in \{1,...,N\}\setminus A_n}\bbE Y_nY_m$.
Let $1\leq n,m\leq N$ be such that $m\not\in A_n$.
Consider the sets of indexes $\Gam_k=\{jn: 1\leq j\leq\ell\}$ where $k=n,m$ and set 
$\Gam_{n,m}=\Gam_n\cup\Gam_m$.
By the definition of the set $A_n$ we have $dist(\Gam_n,\Gam_m)>l$. Therefore, the set $\Gam_{n,m}$ can be 
represented in the form 
\begin{eqnarray*}
\Gam_{n,m}=\bigcup_{t=1}^LB_t, \,\,B_1<B_2<...<B_L
\end{eqnarray*}
 where $L\leq2\ell$, each $B_t$ is either a subset of 
$\Gam_n$ or of $\Gam_m$ and $dist(B_t,B_{t-1})>l$.
Set 
\[
U_t=\{\xi_{s,r}: s\in B_t\},\,\, t=1,...,L.
\]
Since $r\leq\frac l4$, there exist numbers $n_t$ and $m_t$, $t=1,...,L$ such that
$n_{t-1}<m_t\leq n_t<m_{t+1}+\frac l2$, where $n_0=-\infty,m_{L+1}:=\infty$, and
each $U_t$ is measurable with respect to $\cF_{m_t, n_t}$. 
Set $\cC_1=\{1\leq t\leq L: B_t\subset \Gam_n\}$ and 
$\cC_2=\{1\leq t\leq L: B_t\subset \Gam_m\}$. Then $\{\cC_1,\cC_2\}$ is 
a partition of $\{1,2,...,L\}$, $Y_n$ is measurable
with respect to $\sig\{U_t: t\in\cC_1\}$ and $Y_m$ is measurable
with respect to $\sig\{U_t: t\in\cC_2\}$.
 Therefore, by (\ref{Exp4}) and 
(\ref{Y-Bound}) and since $\bbE Y_n=0$, 
\begin{eqnarray*}
|\bbE Y_nY_m|\leq 64\ell N^{-1}D^{-2}\phi(\frac l2)\leq 64d\ell D^{-2}N^{-1}c^{\frac l2}\leq
64d\ell D^{-2}N^{2A\ln c-1}
\end{eqnarray*}
implying that
\begin{equation}\label{gam good D rate}
|\gam|=|\del-1|=|\del-\text{Var}\bar Z_{N,r}|\leq 64d\ell D^{-2}N^{1+2A\ln c}.
\end{equation}
We assume now, in addition to the previous
restriction on $N$, that $64d\ell D^{-2}N^{-\frac12}<\frac12$. Then 
$\del>\frac12$ and so we can set $\sig=\sqrt\del$ and
$W=\frac {\bar Z_{N,r}}\sig$. Then $\sig^2\geq \frac12$ and using (\ref{gam good D rate})
we obtain
\begin{equation}\label{"old (4.11)"}
\|W-\bar Z_{N,r}\|_{L^\infty}\leq\|\bar Z_{N,r}\|_{L^\infty}|1-\frac1\sig|\leq
4\|\bar Z_{N,r}\|_{L^\infty}|\del-1|\leq 16D^{-3}N^{\frac32+2A\ln c}
\end{equation}
where we also used that $\bar s_{N,r}\geq \frac12DN^{\frac12}$.
Since $A|\ln c|>1$ the above right hand side does not exceed
$16D^{-3}N^{-\frac12}$ which together with (\ref{Good apprx 1}) and Lemma \ref{lem3.17}
yields that 
\begin{equation}\label{Kol1}
d_K(\cL(Z_N),\cN(0,1))\leq 3d_K(\cL(W),\cN(0,1))+c_5\max(D^{-1},D^{-3})N^{-\frac12}\ln^2(N+1)
\end{equation}
where $c_5=16c_4$.

In order to approximate $d_K(\cL(W),\cN(0,1))$, 
set $X_n=\sig^{-1}Y_n,\,n=1,2,...,N$. Then $W=\sum_{n=1}^NX_n$ 
and by (\ref{Y-Bound}) we have 
$\|X_n\|_{L^\infty}\leq R$, where $R=4N^{-\frac12}D^{-1}\sig^{-1}\leq 8N^{-\frac12}D^{-1}$.
Applying Theorem 2.1 in \cite{Rin}, using the equality (15) from there and taking into account 
(\ref{A-Bound}) we obtain that
\begin{equation*}
d_K(\cL(W),\cN(0,1))\leq
 R_1+R_2+R_3+K_1lR+2K_1^2l^2NR^3
\end{equation*}
where 
\begin{eqnarray*}
&R_1=4\|\sum_{n=1}^N\sum_{m\in A_n}(X_nX_m-\bbE X_nX_m)\|_2,\hskip2cm\\
&R_2=\sqrt{2\pi}\sum_{n=1}^N\bbE\big|\bbE[X_n|X_m: m\notin A_n]\big|,\\
&\hskip3cmR_3=2\|\sum_{n=1}^NX_n\big(\sum_{m\in A_n}X_m\big)^2\|_2\big(\|W\|_2+5\big)
\end{eqnarray*}
and $\|X\|_q^q=\bbE |X|^q=\|X\|_{L^q}^q$ for any random variable $X$.
Now we estimate $R_1,R_2$ and $R_3$. Set
$T_n=\sum_{m\in A_n}(X_nX_m-\bbE X_nX_m)$, $n=1,...,N$.
Then
\begin{eqnarray*}
R_1^2\leq16\sum_{n_1,n_2=1}^N\bbE T_{n_1}T_{n_2}.
\end{eqnarray*}
Let $n_1$ and $n_2$ be such that $d_\ell(A_{n_1},A_{n_2})=k>2r$. 
Consider the sets $\Gam_s=\{jm: m\in A_{n_s}, 1\leq j\leq\ell\}$, $s=1,2$. Then
$dist (\Gam_1,\Gam_2)=d_\ell(A_{n_1},A_{n_2})=k$. 
 Set $\Gam=\Gam_1\cup\Gam_2$. Both $\Gam_i$'s are  unions of at most $\ell^3$ intervals
 and therefore there exist sets $B_1,B_2,...,B_L$, $L\leq2\ell^3$ such that 
\begin{eqnarray*}
\Gam=\bigcup_{t=1}^LB_t,\,\,\,B_1<B_2<...<B_L
\end{eqnarray*}
where each $B_t$ is either a subset of $\Gam_1$ or a subset of $\Gam_2$ and 
$dist(B_t,B_{t-1})\geq k$, $t=2,...,L$.
Set 
\[
U_t=\{\xi_{a,r}: a\in B_t\},\,\,\, t=1,...,L.
\]
Then there exist numbers $m_t,n_t, t=1,2,...,L$ 
 such that 
$n_{t-1}<m_t\leq n_t\leq m_{t+1}+k-2r$, $n_0=-\infty,m_{L+1}:=\infty$
 and each $U_t$ is $\cF_{m_t,n_t}$-measurable.
 Set $\cC_s=\{1\leq t\leq L: B_t\subset\Gam_s\}$, $s=1,2$. Then $\{\cC_1,\cC_2\}$ is a partition of
 $\{1,2,...,L\}$ and $T_{n_s}$, $s=1,2$ is measurable with respect to $\sig\{U_t: t\in\cC_s\}$.
Since $\|X_n\|_{L^\infty}\leq R$  we have $\|T_n\|_{L^\infty}\leq 2K_1lR^2$ (recalling (\ref{A-Bound}))
and thus by (\ref{Exp4}),
\begin{eqnarray*}
|\bbE T_{n_1}T_{n_2}|\leq 16K_1^2l^2R^4 L^2\phi(k-2r)\leq
64\ell^6K_1^2l^2R^4dc^{k-2r}
\end{eqnarray*}
where we used that $\bbE T_n=0$. Given $n_1$ and $k>2r$, the number of $n_2$'s
satisfying $d_\ell(A_{n_1},A_{n_2})=k$  is at most $K_2l$ (recalling (\ref{A-Bound})), while for any 
other $n_2$ and $k$ we can use the trivial upper bound 
$|\bbE T_{n_1}T_{n_2}|\leq\|T_{n_1}\|_{L^\infty}\|T_{n_2}\|_{L^\infty}\leq4K_1^2l^2R^4$.
Therefore, by the definitions of $R$ and $r$, 
\begin{eqnarray*}
R_1^2\leq64\ell^4K_1^2l^2R^4N
\big(K_2ld\sum_{k=2r+1}^Nc^{k-2r}+(2r+1)K_2l\big)\leq C_0l^4N^{-1}D^{-4}
\end{eqnarray*}
where $C_0$ is a constant which depends only on $c$ and $d$ and $\ell$. In order
to approximate $R_2$, let $1\leq n\leq N$ and set $\mathscr X_n=\{X_m:m\notin A_n\}$. 
Then,
\begin{equation}\label{R2.0}
\|\bbE[X_n|\mathscr X_n]\|_1^2\leq\|\bbE[X_n|\mathscr X_n]\|_2^2=
\bbE X_n\bbE[X_n|\mathscr X_n].
\end{equation}
Consider the sets $\tau_1=\{n,2n,...,\ell n\}$ and $\tau_2=\{jm: m\not\in A_n, 1\leq j\leq \ell\}$. Then
by the definition of $A_n$ we have $dist(\tau_1,\tau_2)>l$. Thus, 
the union $\tau_1\cup\tau_2$ can be written
as a union of at most $2\ell+1$ disjoint sets $B_1,B_2,...,B_L$ such that 
$B_1<B_2<...<B_L$, $dist(B_t,B_{t+1})>l$ and 
each $B_t$ is either a subset of $\tau_1$ of a subset of $\tau_2$. Consider the random vectors 
\[
U_t=\{\xi_{s,r}: s\in B_t\},\,\,t=1,...,L
\]
and the partition of $\{1,2,...,L\}$ into the sets $\{\cC_1,\cC_2\}$, where
 $\cC_s=\{1\leq t\leq L: B_t\subset\tau_s\}$, $s=1,2$.   Then $X_n$ is
measurable with respect to $\sig\{U_t: t\in\cC_1\}$ and $\bbE[X_n|\mathscr X_n]$
is measurable with respect to $\sig\{U_t: t\in\cC_2\}$.
Therefore by (\ref{Exp4}) and (\ref{exp psi}),
\begin{eqnarray}\label{R2.1}
&|\bbE[X_n\bbE[X_n|\mathscr X_n]]|\leq
4(2\ell+1) R^2dc^{\frac l2}
\end{eqnarray}
where we used that $r\leq\frac l4$, $\bbE X_n=0$ and that 
$\|\bbE [X_n|\mathscr X_n]\|_{L^\infty}\leq\|X_n\|_{L^\infty}\leq R$. 
We conclude from (\ref{R2.0}) and (\ref{R2.1}) that there exists a constant $C_0'$ which depends
only on $\ell$ such that
\begin{eqnarray}\label{R2}
R_2\leq C_0'N^{\frac12} D^{-1}d^\frac12c^{\frac l4}.
\end{eqnarray}
To estimate $R_3$, first observe that by the definition of $W$ and by (\ref{Q2})
we have  $\|W\|_2^2=\del^{-1}\|\bar Z_{N,r}\|_2^2=1+\del^{-1}\gam$ and therefore
$\|W\|_2\leq 2$, since $|\gam|<\frac12$ and $\del>\frac12$.
The first factor in the definition of $R_3$ is clearly bounded from above
by $2NK_1^2l^2R^3$ and we conclude that
\begin{eqnarray*}
R_3\leq C_4 l^2D^{-3} N^{-\frac12}
\end{eqnarray*}
for some constant $C_4$ which depends only on $\ell$. The estimate (\ref{Kol-Opt}) in Theorem \ref{thm3.3} follows now by taking any $A>\max(1,2|\ln c|^{-1})$, using (\ref{Kol1}) and the above estimates of $R_i$'s. 
Note that all the approximations in this section hold true only for $N$'s  
satisfying $3N\geq 8D^{-2}(C_1+c_2)$ and 
$64d\ell D^{-2}N^{-\frac12}<\frac12$,
but inequality (\ref{Kol-Opt}) follows for all other $N$'s
from the basic estimate $d_K(\cL(Z_N),\cN(0,1))\leq1$. We also remark that when
$\be_\infty(r_0)=0$
for some $r_0$ then taking $r\geq r_0$ we get $S_{N,r}=S_N$ and so there is no need
 for (\ref{GoodHold1}) to hold true.

Now we derive (\ref{Kol-Opt1.1.}) where again it is sufficient to consider the case when 
$M=1$. 
Let $0<\epsilon<\frac14$. 
First for any $b>1$,
\begin{eqnarray*}
&\|D^{-1}N^{-\frac12}S_N-Z_N\|_{L^b}=\|S_N\|_{L^b}|N^{-\frac12}D^{-1}-(s_N)^{-1}|\\
&=\|S_N\|_{L^b}\Big|\frac{\bbE (S_N)^2-D^2N}{D^2Ns_N+D(s_N)^2N^\frac12}\Big|
\end{eqnarray*}
where in the second equality we used that $|x^{-1}-y^{-1}|=|x^2-y^2|(xy^2+yx^2)^{-1}$ for any 
$x,y>0$.
By  Lemma 5.2 in \cite{HK1} for any $b>1$ there exits $\Gamma_b$
which depends only on $c,d,b$ and $\ell$ so that $\|S_N\|_{L^b}\leq \Gamma_bN^{\frac12}$.
Using  the previous estimates, for any $N$ so that 
$3N^{\frac12}D^2\geq 8(C_1+c_2)$ and $64d\ell D^{-2}N^{-\frac12}<\frac12$
 we have $s_N\geq\frac 12D$. Therefore,
\[
\|D^{-1}N^{-\frac12}S_N-Z_N\|_{L^b}\leq 8D^{-3}C_1\Gamma_bN^{-\frac12}
\]
where and we also used (\ref{D-Rate}). Applying the second statement 
of Lemma \ref{lem3.17} with  $b=\frac1{2\epsilon}-1$ and using (\ref{Kol-Opt}) completes the proof of 
(\ref{Kol-Opt1.1.}).
\qed

\subsection{Extensions}
\subsubsection*{Unbounded functions}
Let $M,\iota>0$, $\ka\in(0,1]$ and  $F:(\bbR^\wp)^\ell\to\bbR$ be a function satisfying
\begin{eqnarray*}
&|F(x)|\leq M(1+\sum_{i=1}^\ell |x_i|^\iota)\, \text{ and}\\
&|F(x)-F(y)|\leq M(1+\sum_{i=1}^\ell |x_i|^\iota+|y_i|^\iota)\sum_{i=1}^\ell|x_i-y_i|^\ka
\end{eqnarray*}
for any $x=(x_1,...,x_\ell)$ and $y=(y_1,...,y_\ell)$ in $(\bbR^\wp)^\ell$. For any $R>0$
set $F_R(x)=F(x)\bbI(|F(x)|\leq R)$.
Then, assuming that for some $p>\iota+1$,
\[
\gam_p=\|\xi_1\|_{L^{\iota p}}<\infty
\]
we can first approximate $F(\xi_n,\xi_{2n},...,\xi_{\ell n})$ by $F_R(\xi_{n,r},\xi_{2n,r}...,\xi_{\ell n,r})$ in the $L^p$-norm and then  use  Lemma 3.3 \ref{lem3.17}.
Applying Theorem \ref{thm3.3} with the function $F_R$ and taking $R$ with an appropriate dependence
on $N$ we obtain convergence rate of the form
$CN^{-\frac12+\ve_p}$, where $\ve_p$ depends on $p$ and satisfies $\lim_{p\to\infty}\ve_p=0$.
In fact, similar type of rates can be obtained assuming only that
$\phi(n)+\beta_q(n)\leq dn^{-\te}$ for some $q,d,\te>0$, where $\be_q$ is defined 
similarly to $\be_\infty$, but with the $L^q$ norm.

\subsubsection*{Nonlinear indexes}
Let $q_i, i=1,...,\ell$ be strictly increasing functions satisfying $q_i(\bbN)\subset\bbN$ 
which are ordered so that 
\[
q_1(n)<q_2(n)<...<q_\ell(n)\,\, \text { for any sufficiently large } n.
\]
Consider the sums 
\begin{eqnarray*}
&S_N=\sum_{n=1}^NF(\xi_{q_1(n)},\xi_{q_2(n)},...,\xi_{q_\ell(n)}).
\end{eqnarray*}
The proof of  Theorem \ref{thm3.3} proceeds essentially in the same way 
when all $q_i$'s are linear. For more general $q_i$'s, set
\begin{eqnarray*}
A_n=\{1\leq m\leq N: \min_{1\leq i,j\leq \ell}|q_i(n)-q_j(m)|\leq l\}.
\end{eqnarray*}
The proof of Theorem \ref{thm3.3}
will proceed similarly for the sums $S_N$ if we show that
limit $D^2=\lim_{N\to\infty} \bbE S_N^2$ exists, obtain convergence rate towards it
and upper bounds similar to the ones in (\ref{A-Bound}).
Suppose that $q_1,...,q_k$ are linear, for some $k<\ell$ and that $q_j, j\geq k$ are not.
When all $q_i$'s are polynomials,  existence of  $D^2$ is proved in \cite{HK2}.
 Though the limit
$D^2$ does not exist in general, if   $q_{j+1}$ grows faster then $q_j$ for $j>k$ in the sense of
(2.11) in \cite{KV}, 
then existence of $D^2$ follows from Theorem 2.3 in \cite{KV}.
Convergence rate towards $D^2$  when $q_i$'s  are polynomials
 can be obtained by proceeding similarly to the proof of Proposition 5.3 in \cite{HK2}. If, instead,
 $q_{j+1}(n^\al)-q_j(n)$ converges to $\infty$ as $n\to\infty$ for some $0<\al<1$
and all $j\geq k$, then convergence rate towards $D^2$ with some dependence
on $\al$ follows from the arguments in \cite{KV}.

Each $q_i(n)$ grows at least as fast as linearly  which implies that $|A_n|$ is of order $l$. 
When  all $q_i$'s are polynomials of the same degree then the limit 
$\lim_{n\to\infty} q_i^{-1}(q_j(n))/n$ exists for any $1\leq i,j\leq\ell$ and therefore the 
proof of the second upper bound in (\ref{A-Bound}) proceeds in a similar
way but with $\tilde d_\ell(a,b)=\min_{1\leq i,j\leq\ell}|q_i(a)-q_j(b)|$ in place of 
$d_\ell(a,b)$. When $q_i$'s do not necessarily have the same degree then 
beginning the summation in the definition of $S_N$ from  $cN^\gam$ for appropriate $\gam<1$ and $c>0$, 
guarantees that $|q_i(n)-q_j(m)|>CN$  when $\deg{q_i}\not=\deg{q_j}$ and 
$cN^\gam\leq n,m\leq N$.
Similar to the latter inequality is satisfied when  $\max(i,j)>k$ and $q_s$ grows faster than
$q_{s-1}$ for $s=k+1,...,\ell$ and so an appropriate version of (\ref{A-Bound})
follows in this situation, as well.

\end{document}